\documentclass[a4paper, 12pt]{amsart}

\usepackage{amsfonts}
\usepackage{amsthm}
\usepackage{amsmath, amssymb}
\usepackage{mathtools} 
\usepackage{enumerate}
\usepackage{graphicx}
\usepackage{color}
\usepackage{enumitem}
\usepackage[percent]{overpic}
\usepackage{relsize} 
\usepackage[colorlinks,linkcolor=blue,citecolor=blue]{hyperref}
\usepackage{thmtools, thm-restate}
\usepackage[percent]{overpic}

\usepackage[top=2.5cm, bottom=2.5cm, left=2.5cm, right=2.5cm]{geometry}
\newtheorem{thm}{Theorem}

\newtheorem{prop}[thm]{Proposition}

\newtheorem{lemma}[thm]{Lemma}

\newtheorem*{claim}{Claim}

\newcommand{\Z}{\mathbb{Z}}
\newcommand{\R}{\mathbb{R}}

\newcommand{\Q}{\mathbb{Q}}

\newcommand{\spin}{\ifmmode{\rm Spin}\else{${\rm spin}$\ }\fi}
\newcommand{\spinc}{\ifmmode{{\rm Spin}^c}\else{${\rm spin}^c$}\fi}

\newcommand{\norm}[1]{\lVert #1 \rVert^2}

\begin{document}

\title{On Seifert fibered spaces bounding definite manifolds}

\author{Ahmad Issa}
\address{Department of Mathematics \\
         The University of Texas At Austin \\
         Austin, TX, 78712, USA}
\email{aissa@math.utexas.edu}

\author{Duncan McCoy}
\address{Department of Mathematics \\
         The University of Texas At Austin \\
         Austin, TX, 78712, USA}
\email{d.mccoy@math.utexas.edu}

\begin{abstract} We establish an inequality which gives strong restrictions on when the standard definite plumbing intersection lattice of a Seifert fibered space over $S^2$ can embed into a standard diagonal lattice, and give two applications. First, we answer a question of Neumann-Zagier on the relationship between Donaldson's theorem and Fintushel-Stern's $R$-invariant. We also give a short proof of the characterisation of Seifert fibered spaces which smoothly bound rational homology $S^1 \times D^3$'s.
\end{abstract}

\maketitle

\section{Introduction}
Donaldson's diagonalization theorem \cite{MR910015} has led to many great successes in understanding several important questions in low dimensional topology, and in knot theory in particular. For example, Donaldson's theorem can often be used to answer questions concerning sliceness, unknotting number, $3$-manifolds bounding rational homology balls, and surgery questions. In these cases, one typically uses Donaldson's theorem to obstruct a certain $3$-manifold from bounding a certain type of smooth negative definite $4$-manifold, with the obstruction taking the form of the existence of a certain map of intersection lattices. However, understanding this obstruction for large families of examples is often highly non-trivial, and can require combinatorial ingenuity.



One appealing application of Donaldson's theorem is to prove the well-known fact that the Poincar\'e homology sphere $P = S^2(2; 2, \frac{3}{2}, \frac{5}{4})$ does not bound a smooth integral homology $4$-ball. This fact can, of course, be proved in many other ways, for example by using Rokhlin's theorem, Fintushel-Stern's $R$-invariant, or by using the $d$-invariant coming from Heegaard Floer homology. Assuming that $P$ is oriented to bound the positive $E_8$ plumbing, the proof by Donaldson's theorem is as follows. If $P$ were the boundary of a smooth integral homology $W$, then we could form a closed positive definite manifold by gluing $-W$ to the positive $E_8$ plumbing. Donaldson's theorem would then imply that the $E_8$ intersection form is diagonalizable, which is, of course, untrue. In fact, as the $E_8$ intersection form does not embed into any diagonal lattice, this arguments shows that $P$ does not bound any smooth negative definite 4-manifold. The purpose of this paper is to generalize this argument to other Seifert fibered spaces. We prove the following theorem.


\begin{restatable}{thm}{sfsineq}
  \label{thm:sfsineq}
  Let $Y = S^2(e; \frac{p_1}{q_1}, \ldots, \frac{p_k}{q_k})$, $k \ge 3$, be a Seifert fibered space over $S^2$ in standard form, that is, with $e > 0$, $\frac{p_i}{q_i} > 1$ for all $i\in\{1,2,\ldots,k\}$ and $\varepsilon(Y) \ge 0$.
  Suppose that $Y$ bounds a smooth $4$-manifold $W$ such that $\sigma(W) = -b_2(W)$ and the inclusion induced map $H_1(Y; \Q) \rightarrow H_1(W; \Q)$ is injective. Then there is a partition of $\{1,2,\ldots,k\}$ into at most $e$ classes such that for each class $C$,
  \[\sum_{i\in C} \frac{q_i}{p_i} \le 1.\]
\end{restatable}

We note that the condition that $\varepsilon(Y) := e - \sum_{i=1}^k \frac{q_i}{p_i} \ge 0$ in Theorem \ref{thm:sfsineq} guarantees that $Y$ is oriented to bound a positive (semi-)definite plumbing $4$-manifold. When $Y$ is a rational homology sphere the map $H_1(Y; \Q) \rightarrow H_1(W; \Q)$ is automatically injective so in this case we are simply obstructing the existence of a negative definite manifold bounding $Y$. Although we do not discuss the details in this paper, one can easily obtain analogous results for Seifert fibered spaces over any orientable base surface. In our notation, the Poincar\'e homology sphere oriented to bound the positive $E_8$ plumbing is $P=S^2(2;2,\frac{3}{2}, \frac{5}{4})$, see Figure \ref{fig:sfs_surgery}. The reader can easily verify that Theorem~\ref{thm:sfsineq} obstructs $P$ from bounding a negative definite manifold. 
Finally, we note that the converse to Theorem \ref{thm:sfsineq} is not true. The integer homology sphere $S^2(1; 3, 5, \frac{13}{6})$ passes the obstruction, but does not bound a negative definite manifold as it bounds a positive definite plumbing whose intersection form does not embed in a diagonal lattice.

We give two applications of Theorem \ref{thm:sfsineq}. First, we prove the following theorem.

\begin{thm}\label{thm:nz_simple}
Let $Y = S^2(e; \frac{p_1}{q_1}, \ldots, \frac{p_k}{q_k})$ be a Seifert fibered integral homology sphere in standard form, that is, with $\frac{p_i}{q_i} > 1$ for all $i\in\{1,2,\ldots,k\}$, $e > 0$ and with $Y$ oriented to bound a smooth positive definite plumbing $4$-manifold. If $Y$ bounds a smooth negative definite $4$-manifold, then $e = 1$.
\end{thm}

In the course of proving Theorem \ref{thm:nz_simple}, we obtain a positive answer to the following question asked by Neumann-Zagier \cite{MR827273}.

{\bf Question:} \emph{Let $Y$ be as in Theorem \ref{thm:nz_simple}. If the intersection form of the plumbing of $Y$ is diagonalizable over $\Z$, must $e$ be equal to $1$?}


The motivation for this question comes from the $R$-invariant. 
Fintushel-Stern \cite{MR808222} used gauge theory to define an invariant $R(Y)$ of Seifert fibered integral homology spheres with the property that if $R(Y) > 0$ then $Y$ does not bound a smooth negative definite $4$-manifold $W$ with $H_1(W)$ having no $2$-torsion. 
Fintushel-Stern originally gave an expression for $R(Y)$ as a trigonometric sum involving the Seifert invariants of $Y$. Neumann-Zagier \cite{MR827273} proved that these sums could be simply evaluated in terms of the central weight $e$ of the standard positive definite plumbing bounding $Y$, showing that $R(Y) = 2e - 3$. Thus, if $e > 1$ then the $R$-invariant shows that $Y$ does not bound a smooth negative definite $4$-manifold $W$ with $H_1(W)$ having no $2$-torsion. In this light, the positive answer to Neumann-Zagier's question implies that this result obtained from the $R$-invariant is also a consequence of Donaldson's theorem.


We are in fact able to prove a more general version of Theorem \ref{thm:nz_simple} which holds for all $|H_1(Y)| \in \{1,2,3,5,6,7\}$, see Theorem \ref{thm:nz_conj} of Section \ref{sec:nz}. Some particular cases of Theorem \ref{thm:nz_simple} are known. In their original paper, Neumann-Zagier \cite{MR827273} claimed to have proved the cases when $k = 3$, and when $k=4$ and $e \neq 3$, but do not provide a proof, remarking that their proof was ``clearly not the right proof''. The special case when $e = k-1$ follows from \cite[Lemma 3.3]{MR2782538}.

Finally, we note that a positive answer to Neumann-Zagier's question is a special case of a more general conjecture made by Neumann \cite{MR1012837}, stating that if an integral homology sphere $Y$ is given as the boundary of a positive definite plumbing tree $\Gamma$ and the intersection lattice of $\Gamma$ is isomorphic to a diagonal lattice, then some vertex of $\Gamma$ has weight $1$. This general form of Neumann's conjecture for graph manifolds remains open.

Lidman-Tweedy \cite[Remark 4.3]{1707.09648} asked whether a Seifert fibered integral homology sphere with central weight different from $1$ must have non-vanishing Heegaard-Floer $d$-invariant. As a corollary of Theorem \ref{thm:nz_simple}, we answer their question positively.

\begin{restatable}{corol}{coroldinv}\label{cor:lidmantweedy} Let $Y$ be a Seifert fibered integral homology sphere, and let $e\in \Z$ be the central weight in the standard definite plumbing graph for $Y$. If $|e| \neq 1$, then $d(Y) \neq 0$.
\end{restatable}

As a second application, we give a short proof of the following theorem which, in particular, gives a classification of the Seifert fibered spaces bounding rational homology $S^1 \times D^3$'s.

\begin{restatable}{thm}{paolosthm}\label{thm:bd_qhs1xd3} Let $Y$ be a Seifert fibered space over $S^2$ with $H_*(Y;\Q) \cong H_*(S^1\times S^2;\Q)$. The following are equivalent:
  \begin{enumerate}
    \item\label{enum:comppairs} $Y$ is of the form $S^2(m; \frac{p_1}{q_1}, \frac{p_1}{p_1-q_1}, \ldots, \frac{p_m}{q_m}, \frac{p_m}{p_m - q_m})$, where $m \ge 0$ and $\frac{p_i}{q_i} > 1$ for all $i\in\{1,\ldots,m\}$.
    \item\label{enum:bdqhs1d3} $Y = \partial W$, where $W$ is a smooth $4$-manifold with $H_*(W; \Q) \cong H_*(S^1 \times D^3; \Q)$.
    \item\label{enum:bdposneg} $Y$ is the boundary of smooth $4$-manifolds $W_+$ and $W_-$ such that $\sigma(W_\pm) = \pm b_2(W_\pm)$ and each of the inclusion-induced maps $H_1(Y; \Q) \rightarrow H_1(W_\pm; \Q)$ is injective.
  \end{enumerate}
\end{restatable}


Seifert fibered spaces bounding rational homology $S^1\times D^3$'s naturally arise in two contexts. First, a Seifert fibered space rational homology $S^1\times S^2$ which embeds in $S^4$ necessarily bounds a rational homology $S^1 \times D^3$. Indeed, in this context Donald \cite[Proof of Theorem 1.3]{MR3271270} proved the implication \eqref{enum:bdqhs1d3} implies \eqref{enum:comppairs} of Theorem \ref{thm:bd_qhs1xd3}.
Second, a smoothly slice $2$-component Montesinos link has double branched cover a Seifert fibered space over $S^2$ bounding a rational homology $S^1 \times D^3$. Motivated by trying to determine the slice $2$-component Montesinos links, Aceto \cite[Theorem 1.2]{1502.03863} also classified Seifert fibered spaces bounding rational homology $S^1\times D^3$'s.

Much like the proofs by Donald and Aceto, our proof also proceeds by means of Donaldson's theorem. However, their proofs rely on the work of Lisca \cite{MR2366190} which gives a detailed analysis on sums of linear lattices embedding in a full-rank lattice. We give a short proof of Theorem \ref{thm:bd_qhs1xd3} circumventing the reliance on Lisca's work. 
We obtain the additional equivalent condition \eqref{enum:bdposneg} in Theorem \ref{thm:bd_qhs1xd3}, since our method does not require the lattice embeddings to have full-rank.




Finally, we note that Theorem \ref{thm:sfsineq} also plays a key role in a forthcoming paper \cite{s4paper}, where we analyse which Seifert fibered spaces smoothly embed in $S^4$, and in particular, completely determine the Seifert fibered spaces $Y = S^2(e; r_1,\ldots, r_k)$ with $r_i \in \Q_{>1}$ for all $i$, $\varepsilon(Y) > 0$ and $e > k/2$ which smooothly embed in $S^4$.





In Section \ref{sec:defs}, we recall some standard facts and establish notation and conventions. In Section \ref{sec:embineq}, we prove the key technical theorem used to prove Theorem \ref{thm:sfsineq}. In Section \ref{sec:sfsineq}, we analyse when gluing compact $4$-manifolds with boundary results in a definite $4$-manifold and give a proof of Theorem~\ref{thm:sfsineq}. In Section \ref{sec:nz}, we prove Theorem \ref{thm:nz_conj} answering Neumann-Zagier's question, as well as prove Corollary \ref{cor:lidmantweedy}. Finally, in Section \ref{sec:bding_qhs1xd3} we prove Theorem~\ref{thm:bd_qhs1xd3} determining the Seifert fibered spaces which bound rational homology $S^1\times D^3$'s.

\subsection*{Acknowledgements}
The first author would like to thank Cameron Gordon for his support and encouragement, and Josh Greene for a helpful conversation on Neumann-Zagier's question.

\section{Preliminaries}\label{sec:defs}
In this section we briefly recall some standard facts about Seifert fibered spaces and intersection lattices, as well as establish notation and conventions. See \cite{MR518415} for a more indepth treatment on Seifert fibered spaces and plumbings.

Given $r \in \Q_{>1}$, there is a unique (negative) continued fraction expansion
$$r = [a_1, \ldots, a_{n}]^- := a_1 - \cfrac{1}{a_2 - \cfrac{1}{\begin{aligned}\ddots \,\,\, & \\[-3ex] & a_{n-1} - \cfrac{1}{a_{n}} \end{aligned}}},$$
where $n \ge 1$ and $a_i \ge 2$ are integers for all $i \in \{1, \ldots, n\}$. We associate to $r$ the weighted linear graph (or linear chain) given in Figure \ref{fig:linearchain}. We call the vertex with weight labelled by $a_i$ the $i$th vertex of the linear chain associated to $r$, so that the vertex labelled with weight $a_1$ is the first, or starting vertex of the linear chain.

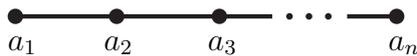
\begin{figure}[h]
  \begin{overpic}[width=150pt]{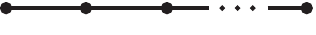}
    \put (0, 0) {$a_1$}
    \put (24, 0) {$a_2$}
    \put (50, 0) {$a_3$}
    \put (95, 0) {$a_n$}
  \end{overpic}
  \caption{Weighted linear chain representing $r = [a_1,\ldots,a_n]^-$.}
  \label{fig:linearchain}
\end{figure}

We denote by $Y = S^2(e; \frac{p_1}{q_1}, \ldots, \frac{p_k}{q_k})$ the Seifert fibered space over $S^2$ given in Figure \ref{fig:sfs_surgery}, where $e\in\Z$, and $\frac{p_i}{q_i} \in \Q$ is non-zero for all $i \in \{1,\ldots,k\}$. The generalised Euler invariant of $Y$ is given by $\varepsilon(Y) = e - \sum_{i=1}^k \frac{q_i}{p_i}$. Every Seifert fibered space $Y$ is (possibly orientation reversing) homeomorphic to one in standard form, i.e. such that $\varepsilon(Y) \ge 0$ and $\frac{p_i}{q_i} > 1$ for all $i \in \{1,\ldots,k\}$. We henceforth assume that $Y$ is in standard form. If $\varepsilon(Y) \neq 0$ then $Y$ is a rational homology sphere with $|H_1(Y)| = |p_1\cdots p_k \varepsilon(Y)|$, and if $\varepsilon(Y) = 0$ then $Y$ is a rational homology $S^1\times S^2$. 

\begin{figure}[h]
  \begin{overpic}[height=80pt]{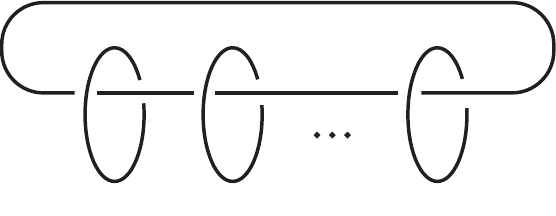}
    \put (-6, 29) {$e$}
    \put (17.5, 0) {$\frac{p_1}{q_1}$}
    \put (39.3, 0) {$\frac{p_2}{q_2}$}
    \put (76.3, 0) {$\frac{p_k}{q_k}$}
  \end{overpic}
  \caption{Surgery presentation for the Seifert fibered space $S^2(e; \frac{p_1}{q_1}, \ldots, \frac{p_k}{q_k})$.}
  \label{fig:sfs_surgery}
\end{figure}
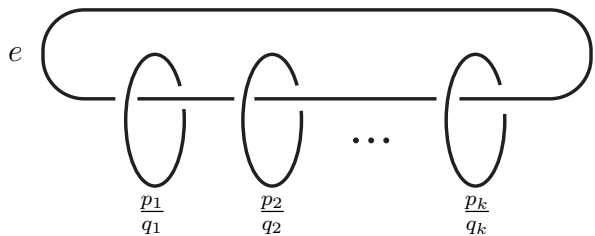

For each $i\in\{1,\ldots, k\}$, we have the unique continued fraction expansion $\frac{q_i}{p_i} = [a_1^i, \ldots, a_{h_i}^i]^-$ where $h_i \ge 1$ and $a_j^i \ge 2$ are integers for all $j \in \{1,\ldots,h_i\}$. We associate to $Y = S^2(e; \frac{p_1}{q_1}, \ldots, \frac{p_k}{q_k})$ the weighted star-shaped graph in Figure \ref{fig:plumbing}. The $i$th leg of the star-shaped graph is the weighted linear subgraph for $p_i/q_i$ generated by the vertices labelled with weights $a_1^i, \ldots, a_{h_i}^i$. The degree $k$ vertex labelled with weight $e$ is called the central vertex.

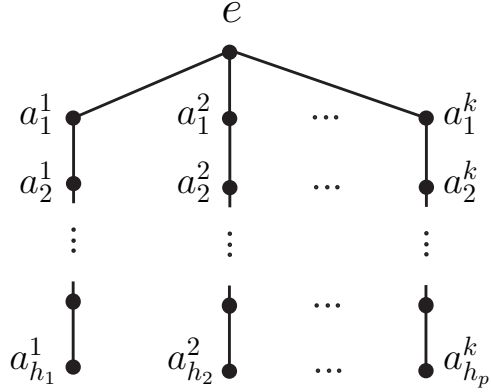
\begin{figure}[h]
  \begin{overpic}[height=150pt]{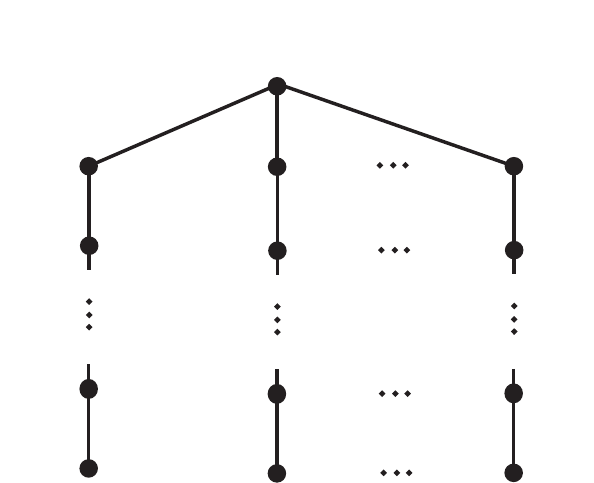}
    \put (45, 73) {\Large $e$}
    \put (4, 52) {\large $a_1^1$}
    \put (4, 38) {\large $a_2^1$}
    \put (2, 2) {\large $a_{h_1}^1$}
    
    \put (36, 52) {\large $a_1^2$}
    \put (36, 38) {\large $a_2^2$}
    \put (34, 2) {\large $a_{h_2}^2$}
    
    \put (90, 52) {\large $a_1^k$}
    \put (90, 38) {\large $a_2^k$}
    \put (90, 2) {\large $a_{h_p}^k$}    


  \end{overpic}
  \caption{The weighted star-shaped plumbing graph $\Gamma$.}
  \label{fig:plumbing}
\end{figure}

Let $\Gamma$ be either the weighted star-shaped graph for $Y$, or a disjoint union of weighted linear graphs. There is an oriented smooth $4$-manifold $X_\Gamma$ given by plumbing $D^2$-bundles over $S^2$ according to the weighted graph $\Gamma$. We denote by $|\Gamma|$ the number of vertices in $\Gamma$. Let $m = |\Gamma|$ and denote the vertices of $\Gamma$ by $v_1, v_2, \ldots, v_m$. The zero-sections of the $D^2$-bundles over $S^2$ corresponding to each of $v_1, \ldots, v_m$ in the plumbing together form a natural spherical basis for $H_2(X_\Gamma)$. With respect to this basis, which we call the vertex basis, the intersection form of $X_\Gamma$ is given by the weighted adjacency matrix $Q_\Gamma$ with entries $Q_{ij}$, $1 \le i,j \le m$ given by

$$Q_{ij} = \begin{cases} 
      \text{w}(v_i), & \mbox{if }i = j \\
      -1, & \mbox{if }v_i\mbox{ and }v_j\mbox{ are connected by an edge} \\
      0, & \mbox{otherwise} 
   \end{cases},
$$
where $\text{w}(v_i)$ is the weight of vertex $v_i$. Denoting by $Q_X$ the intersection form of $X$, we call $(H_2(X), Q_X) \cong (\Z^m, Q_\Gamma)$ the intersection lattice of $X_{\Gamma}$ (or of $\Gamma$). We denote the intersection pairing of two elements $x,y\in\Z^m$ by $x\cdot y = x^T\, Q_\Gamma\, y$. Now assume that $\Gamma$ is the star-shaped plumbing for $Y$. If $\varepsilon(Y) > 0$ then $X_\Gamma$ is a positive definite $4$-manifold and $\Gamma$ is the standard positive definite plumbing graph for $Y$. If $\varepsilon(Y) = 0$, then $X_\Gamma$ is a positive semi-definite manifold. 

Let $\iota : (\Z^m, Q_\Gamma) \rightarrow (\Z^r, \mbox{Id})$, $r > 0$, be a map of lattices, i.e. a $\Z$-linear map preserving pairings, where $(\Z^r, \mbox{Id})$ is the standard positive diagonal lattice. We call $\iota$ a lattice embedding if it is injective. We adopt the following standard abuse of notation. First, for each $i\in\{1,\ldots,m\}$, we identify the vertex $v_i$ with the corresponding $i$th basis element of $(\Z^m, Q_\Gamma)$. Moreover, we shall identify an element $v \in (\Z^m, Q_\Gamma)$ with its image $\iota(v) \in (\Z^r, \mbox{Id})$.

\section{The embedding inequality}\label{sec:embineq}
\label{sec:embineq}
In this section, we prove Theorem \ref{thm:lattice_ineq} below, which is the key technical result of this paper. In particular, it will be used in the next section to prove Theorem \ref{thm:sfsineq}. We begin with some continued fraction identities which we will need.

\begin{lemma}\label{lem:cont_frac_identities}
Let $\{a_i\}_{i\ge 1}$ be a sequence of integers with $a_i \ge 2$ for all $i$, and let $p_k/q_k = [a_1,\dots, a_k]^-$ for all $k \ge 1$. Then we have the following identities:
\begin{enumerate}[label=(\alph*)]
\item $q_n p_{n-1} - p_n q_{n-1}=1$ for all $n \ge 2$.
\item $[a_1, \dots, a_n,x]^- = \frac{xp_n-p_{n-1}}{xq_n - q_{n-1}}$, for all $n \ge 2$ and $x \in \R$ such that both sides are well defined.
\item $p_n= \det \begin{pmatrix}
    a_1 & -1    & 0 & 0 \\
    -1 & a_2 & -1 & 0 \\ 
 0  &-1  & \ddots & -1  \\
 0  & 0 & -1    &a_n      
\end{pmatrix}$ and $q_n= \det \begin{pmatrix}
    a_2 & -1    & 0 & 0 \\
    -1 & a_3 & -1 & 0 \\ 
 0  &-1  & \ddots & -1  \\
 0  & 0 & -1    &a_n      
\end{pmatrix}$ for all $n \ge 2$.
\end{enumerate}
\end{lemma}
\begin{proof}
For $a\in\R$, let $M_a$ denote the matrix
$M_a=\begin{pmatrix}a & -1 \\ 1 & 0\end{pmatrix}$. If $q/r=[a_2,\dots, a_n]^-$, then $\frac{p_n}{q_n} = a_1 - \frac{r}{q}$. In particular, we have
\[
\begin{pmatrix}
p_n \\ q_n
\end{pmatrix}
= M_{a_1}
\begin{pmatrix}
q \\ r
\end{pmatrix}
\]
Thus, one can inductively show that
\begin{equation}\label{eq:cont_frac_matrix}
\begin{pmatrix}
p_n \\ q_n
\end{pmatrix}
= M_{a_1} \dotsb M_{a_n}
\begin{pmatrix}
1 \\ 0
\end{pmatrix},
\end{equation}
and furthermore that
\begin{equation}\label{eq:cont_frac_matrix2}
\begin{pmatrix}
p_n & -p_{n-1} \\ q_n & -q_{n-1}
\end{pmatrix}
= M_{a_1} \dotsb M_{a_n}.
\end{equation}
Identity $(a)$ follows by taking determinants of \eqref{eq:cont_frac_matrix2} and observing that $\det M_{a}=1$ for any $a$. Identity $(b)$ follows from combining \eqref{eq:cont_frac_matrix} and \eqref{eq:cont_frac_matrix2} to get
\[
M_{a_1} \dotsb M_{a_n} M_x \begin{pmatrix}
1 \\ 0\end{pmatrix} = \begin{pmatrix}
p_n & -p_{n-1} \\ q_n & -q_{n-1}
\end{pmatrix}\begin{pmatrix}
x \\ 1 
\end{pmatrix}= \begin{pmatrix}
xp_n -p_{n-1} \\ xq_n - q_{n-1}
\end{pmatrix}.
\]
The identities in $(c)$ can easily be proven by induction using the observation that
\[ \det \begin{pmatrix}
a_1 & -1    & 0 \\
-1  &\ddots &-1  \\
 0  & -1    &a_n      
\end{pmatrix} = a_1 \det \begin{pmatrix}
a_2 & -1    & 0 \\
-1  &\ddots &-1  \\
 0  & -1    &a_n      
\end{pmatrix} -
\det \begin{pmatrix}
a_3 & -1    & 0 \\
-1  &\ddots &-1  \\
 0  & -1    &a_n      
\end{pmatrix}.\]
\end{proof}

The following theorem is the key technical result of this paper.
\begin{restatable}{thm}{latticeineq}
  \label{thm:lattice_ineq}
  Let $\iota : (\Z^{|\Gamma|}, Q_\Gamma) \rightarrow (\Z^m, \mbox{Id})$ be a lattice embedding, where $m > 0$ and $\Gamma$ is a disjoint union of weighted linear chains representing fractions $\frac{p_1}{q_1}, \ldots, \frac{p_n}{q_n} \in \Q_{>1}$. Suppose that there is a unit vector $w \in (\Z^m, \mbox{Id})$ which pairs non-trivially with (the image of) the starting vertex of each linear chain. Then
  \[\sum_{i=1}^n \frac{q_i}{p_i} \leq 1.\]
 Moreover, if we have equality then $w$ has pairing $\pm 1$ with the starting vertex of each linear chain.
\end{restatable}

\begin{proof} Let $\{e_1,\ldots, e_m\}$ denote the orthonormal basis of coordinates vectors of $(\Z^m, \mbox{Id})$. Since the unit vectors in $(\Z^m, \mbox{Id})$ are precisely those vectors of the form $\pm e_i$ where $i\in\{1,\ldots,m\}$, by a change of basis if necessary, we may assume that $w = e_1 \in (\Z^m, \mbox{Id})$. Write $\iota : (\Z^{|\Gamma|}, Q_\Gamma) \rightarrow (\Z^m, \mbox{Id})$ as an integer matrix with respect to the vertex basis of $(\Z^{|\Gamma|}, Q_\Gamma)$, and let $M$ be the transpose of this matrix. Since $\iota$ preserves intersection pairings we have, $u^T Q_\Gamma v = \iota(u)^T \iota(v) = (M^T u)^T (M^T v) = u^T MM^T v$ for all $u, v \in (\Z^{|\Gamma|}, Q_\Gamma)$. Thus,
\[MM^T= Q_\Gamma =
\begin{pmatrix}
A_1 & 0 & 0 \\
0	& \ddots & 0 \\
0	& 	0	 & A_n
\end{pmatrix},
\]
where for each $k \in \{1,\ldots,n\}$, $A_k$ on the diagonal represents a block matrix of the form
\[
A_k = \begin{pmatrix}
    a_1 & -1    & 0 & 0 \\
    -1 & a_2 & -1 & 0 \\ 
 0  &-1  & \ddots & -1  \\
 0  & 0 & -1    &a_l
\end{pmatrix}
\]
where $[a_1, \dots, a_l]^-$ is the continued fraction expansion for $p_k/q_k$. If a matrix $A$ can be written as a product $M' M'^T$, then\footnote{Let $v\ne 0$ be an eigenvector of $A$ with eigenvalue $\lambda$. We have
$v^T A v = \lambda\norm{v}= \norm{M'^Tv}\geq 0,$ thus $\lambda \geq 0$. Since $\det A$ is a product of eigenvalues, this implies that $\det A \geq 0$ as required.}
\begin{equation}\label{eq:posdet}
  \det A \geq 0.
\end{equation}
We will prove the theorem by applying \eqref{eq:posdet} to a matrix of the form $A=M'M'^T$, where $M'$ is a suitable modification of $M$.

We may write $M$ in the form
\[M=
\begin{pmatrix}
M_1 \\
\vdots \\
M_n
\end{pmatrix}\]
where for all $k\in\{1,\ldots,n\}$, $M_k$ is a matrix such that $M_k M_k^T = A_k$. By the assumption that $e_1$ pairs non-trivially with each of the starting vertices of the linear chains, we may assume that each matrix $M_k$ is non-zero in its top left entry. For each $k\in\{1,\ldots,n\}$, let $p_k/q_k = [a_1,\dots, a_l]^-$ be the standard continued fraction expansion and choose $M_k'$ to be the submatrix of $M_k$ obtained by taking the first $l_k$ rows, where $l_k$ is chosen so that the first column $w_k$ of $M_k'$ takes one of the two forms:
\begin{enumerate}[label=\textnormal{(Form \arabic*)}]
\item \label{enum:form_1} $w_k = \begin{pmatrix}u & 0 & \cdots & 0 & v \end{pmatrix}^T$,
 where $l_k > 2$ or $a_1 > 2$, and $v = 0$ only if $M_k' = M_k$.
\item \label{enum:form_2} $w_k = \pm \begin{pmatrix}1 & -1 & \cdots & 0 & v \end{pmatrix}^T$, where $a_1 = 2$, and $v = 0$ only if $M_k' = M_k$.
\end{enumerate}

Let $M'$ be the matrix
\[M'=
\begin{pmatrix}
1 \, 0 \, \cdots\, 0 \\
 M_1' \\
 \vdots \\
 M_n'
\end{pmatrix}.
\]
Then the product $A= M' M'^T$ takes the form of the block matrix
\[M'M'^T=
\begin{pmatrix}
1   & w_1^T &       & w_n^T \\
w_1 & A_1'   &         & 0 \\
 	&       & \ddots  & \\
w_n	& 	0   &	 		& A_n'
\end{pmatrix}.
\]

\begin{claim}
$\det A$ can be written in the form
\[
\det A = (P_1 \dotsb P_n) (1- \sum_{i=1}^n \frac{Q_i}{P_i}),
\]
where $P_k = \det A_k'$ and
$Q_k=-\det
\begin{pmatrix}
0   & w_k^T \\
w_k & A_k' 
\end{pmatrix}$ is a quantity depending only on $A_k$ and $w_k$. 
\end{claim}
\begin{proof}
By the multi-linearity of the determinant in both rows and columns we have
\[
\det
\begin{pmatrix}
1   & w_1^T &  \hdots       & w_n^T \\
w_1 & A_1'   &         & 0 \\
\vdots 	&       & \ddots  & \\
w_n	& 	0   &	 		& A_n'
\end{pmatrix}
=
\det
\begin{pmatrix}
1   & 0 &  \hdots       & 0 \\
0 & A_1'   &         & 0 \\
\vdots 	&       & \ddots  & \\
0	& 	0   &	 		& A_n'
\end{pmatrix}
+ \sum_{1\leq i,j \leq n}
\det B_{ij},
\]
where $B_{ij}$ is the matrix
\[ B_{ij}=
\begin{pmatrix}
0      &\dotsb &  w_j^T &  \dotsb & 0 \\
\vdots & A_1'  &        &         &   \\
w_i    &       & \ddots &   0     &    \\
\vdots &       &   0    & \ddots  &    \\
0	   & 	   &	 	&         & A_n'
\end{pmatrix}.
\]
By cofactor expansion in the first row then first column, it is not hard to check that $\det B_{ij} = 0$ for all $i\ne j$. For $i\in\{1,\ldots,n\}$, by row and column operations, we can put $B_{ij}$ into the form of a diagonal block matrix with diagonal blocks $\begin{pmatrix}
0   & w_k^T \\
w_k & A_k' 
\end{pmatrix}$, $A_1', \ldots, A_{i-1}', A_{i+1}', \ldots, A_n'$ without changing the determinant. Hence, $\det B_{ii}$ is the product of the determinants of these blocks, that is, $\det B_{ii} = -(P_1 \dotsb P_n) \frac{Q_i}{P_i}$. 
\end{proof}
Since $P_k>0$ for all $k$, the previous claim combined with $\det{A} \ge 0$ (see \eqref{eq:posdet}) shows that
\begin{equation}\label{eqn:qpIneq}
\sum_{i=1}^n \frac{Q_i}{P_i} \leq 1.
\end{equation}
So to prove the inequality in the theorem it suffices to show that
$Q_k / P_k \geq q_k /p_k$ for each $k \in \{1,\ldots,n\}$. To do this it suffices to consider some fixed $k \in \{1,\ldots,n\}$. For convenience, let $P/Q=P_k/Q_k$ and $p/q = p_k/q_k = [a_1, a_2, \ldots, a_h]^-$ where $a_i \ge 2$ for all $i \in \{1,\ldots, h\}$, and let $l = l_k$ be the number of rows of $A'_k$.

Consider the following identity obtained by adding the second row to the first row, and the second column to the first column:
\[\det\begin{pmatrix}
0      & -1 & 1   & \dotsb & v \\
-1     & 2  & -1  &	       & 0\\
1      & -1 & a_2 & -1     &  \\
\vdots &    &-1   & \ddots &-1  \\
v      & 0  &     & -1     &a_l
\end{pmatrix}
=
\det\begin{pmatrix}
0      & 1  & 0   & \dotsb & v \\
1      & 2  & -1  &	       & 0\\
0      & -1 & a_2 & -1     &  \\
\vdots &    &-1   & \ddots &-1  \\
v      & 0  &     & -1     &a_l
\end{pmatrix}.
\]
Recall that $w_k$ takes one of two possible forms. By applying the above identity if $w_k$ takes the form \ref{enum:form_2}, we see that regardless of the form that $w_k$ takes, $Q$ is equal to the determinant of a matrix of the following form
\begin{align}\label{eq:gen_form}Q=-\det \begin{pmatrix}
0      & u  & 0   & \dotsb & v \\
u      & a_1  & -1  &	       & 0\\
0      & -1 & a_2 & -1     &  \\
\vdots &    &-1   & \ddots &-1  \\
v      & 0  &     & -1     &a_l
\end{pmatrix},
\end{align}
where if $(u,v) = (\pm1,\mp 1)$ then either $l > 2$ or $a_1 > 2$. If $w_k$ takes the form \ref{enum:form_2}, we define $u\in\{\pm 1\}$ via Equation \eqref{eq:gen_form} by applying the identity.

For $i\in\{1,\ldots,h\}$, let $r_i/s_i$ denote the continued fraction $[a_1, \dots, a_i]^-$.
Note that $P=r_l$.
\begin{claim}
\[Q= u^2 s_l + 2uv + v^2 r_{l-1}\]
\end{claim}
\begin{proof} Applying cofactor expansion along the first column and first row in (\ref{eq:gen_form}) gives
$$Q = u^2 C_1 + (-1)^{l+1} uv C_2 + (-1)^{l+1} uv C_3 + v^2 C_4,$$ where
$C_1 = \det \begin{pmatrix}
a_2  & -1  &  \cdots    & 0\\
-1 & a_3 & -1     &  \\
\vdots    &-1   & \ddots &-1  \\
0  &     & -1     &a_l
\end{pmatrix}$,
$C_2 = \det \begin{pmatrix}
-1  & a_2  & -1	&  \cdots     & 0\\
0 & -1 & a_3  & -1  &  \\
\vdots    &  & -1  & \ddots &-1  \\
  &   &  & -1     &a_{l-1} \\
0 & & & & -1
\end{pmatrix}$, \newline
$C_3 = \det \begin{pmatrix}
-1  & 0  & 	&  \cdots   & 0\\
a_2 & -1 &   &   &  \\
-1 & a_3 & -1  & &  \\
  & -1  & \ddots & -1  & \\
0 & & & a_{l-1} & -1
\end{pmatrix}$
and
$C_4 = \det \begin{pmatrix}
a_1  & -1  & \cdots  & 0\\
-1 & a_2 & -1     &  \\
\vdots    &-1   & \ddots &-1  \\
0  &     & -1     &a_{l-1}
\end{pmatrix}.$

Using the continued fraction identities in Lemma \ref{lem:cont_frac_identities}, we see that $C_1 = s_l$ and $C_4 = r_{l-1}$. Finally, notice that $C_2$ (resp. $C_3$) is the determinant of an upper (resp. lower) triangular matrix with $l-1$ diagonal entries all of which are $-1$, hence $C_2 = C_3 = (-1)^{l-1}$.
\end{proof}
\begin{claim}
We have $\frac{Q}{P} \geq \frac{q}{p}$ with equality only if $u=\pm 1$ and $v=0$.
\end{claim}
\begin{proof}
Recall that if $v=0$ then $r_l/s_l = p/q$. Thus if $v=0$, $Q/P = u^2q/p$. Since $u\ne 0$, we clearly have $Q/P \geq q/p$ with equality only if $u^2=1$, as required.
Thus assume that $v\ne 0$. In this case, if $l = h$, or equivalently, $p/q = r_l/s_l$ then $P = p$ and $Q = u^2 s_l + 2 u v + v^2 r_{l-1} > s_l = q$ and thus $Q/P > q/p$. Hence, we assume that $p/q\ne r_l/s_l$ and, in particular, that $p/q= [a_1, \dots, a_l,x]^-$ where $x = [a_{l+1}, \ldots, a_h]^- >1$. Thus, by Lemma \ref{lem:cont_frac_identities} we have
\[
\frac{p}{q}= \frac{xr_l - r_{l-1}}{xs_l -s_{l-1}}.
\]
Note that since $u$ and $v$ are both non-zero we have that
\[ Q= (s_l-1)u^2 + (u+v)^2 + (r_{l-1}-1)v^2 \geq s_l + r_{l-1} -\varepsilon,\]
where we take $\varepsilon =2$ if $u=-v \in \{\pm 1\}$ and $\varepsilon=1$ otherwise. Note that if $r_{l-1}=2$, then $\varepsilon=1$, as $l=a_1=2$ implies we cannot have $u=-v \in \{\pm 1\}$ by the condition stated immediately following Equation \eqref{eq:gen_form}. In either case we always have
\[r_{l-1}-\varepsilon \geq 1.\]
Thus we obtain
\begin{align}\begin{split}\label{eq:frac_bound}
\frac{Q}{P}-\frac{q}{p}
&\geq \frac{s_l+r_{l-1}-\varepsilon}{r_l} -\frac{xs_l-s_{l-1}}{xr_l-r_{l-1}}\\  
&=\frac{(r_{l-1}-\varepsilon)(xr_l-r_{l-1})+r_l s_{l-1} - s_l r_{l-1}}{r_l(xr_l - r_{l-1})}\\
&=\frac{(r_{l-1}-\varepsilon)(xr_l-r_{l-1})-1}{r_l(xr_l -r_{l-1})}\\
&\geq \frac{(xr_l-r_{l-1})-1}{r_l(xr_l -r_{l-1})}\\
&>0,
\end{split}\end{align}
where we used the identity $r_l s_{l-1} - s_l r_{l-1}=-1$ from Lemma \ref{lem:cont_frac_identities} to obtain the third line, $r_{l-1}-\varepsilon\geq 1$ to obtain the fourth line, and finally that $xr_l-r_{l-1}>1$ which follows by combining $r_l \ge r_{l-1} + 1$ and $x > 1$. This gives the desired inequality, proving the claim.
\end{proof}
The claim together with (\ref{eqn:qpIneq}) proves that $\sum_{i=1}^n \frac{q_i}{p_i} \le 1$ with equality only if $w = e_1$ has pairing $\pm 1$ with each starting vertex, which completes the proof.
\end{proof}

\section{Definite $4$-manifolds and the Seifert fibered space inequality}\label{sec:sfsineq}
Now we consider when gluing two 4-manifolds can result in a closed definite 4-manifold.
\begin{prop}\label{thm:def_gluing_thm}
Let $U_1$ and $U_2$ be 4-manifolds with $\partial U_1 = -\partial U_2 =Y$. Then the closed 4-manifold $X=U_1 \cup_Y U_2$ is positive definite if and only if
\begin{enumerate}[label=(\alph*)]
\item\label{enum:injcond} the inclusion-induced map $(i_1)_* \oplus (i_2)_* \colon H_1(Y; \Q) \rightarrow H_1(U_1; \Q)\oplus H_1(U_2; \Q)$ is injective and
\item\label{enum:sigeq} for $i=1,2$, $U_i$ has the maximal possible signature, that is, \[\sigma(U_i) = b_2(U_i) + b_1(U_i) - b_3(U_i) - b_2(Y).\]
\end{enumerate}  
\end{prop}
\begin{proof}
In this proof all homology groups will be taken with rational coefficients. First, for $i=1,2$, consider the following segment of the long exact sequence in homology of the pair $(U_i, Y)$:
\begin{equation}
0\rightarrow H_3(U_i) \rightarrow H_3(U_i,Y) \rightarrow H_2(Y) \rightarrow H_2(U_i)
\end{equation}
By exactness and Lefschetz duality this shows that the rank of the map $H_2(Y) \rightarrow H_2(U_i)$ is $b_2(Y)-b_1(U_i) +b_3(U_i)$. Homology classes in the image of $H_2(Y) \rightarrow H_2(U_i)$ pair trivially with all classes in $H_2(U_i)$. This gives an upper bound on the signature of $U_i$:
\begin{equation}\label{eq:sigbound}
\sigma(U_i) \leq b_2(U_i) + b_1(U_i) - b_3(U_i) - b_2(Y).
\end{equation}
Now consider the segment of the Mayer-Vietoris sequence
\begin{align}\label{eq:MV}
\begin{split}
0\rightarrow & H_3(U_1)\oplus H_3(U_2) \rightarrow H_3(X) \rightarrow H_2(Y) \rightarrow H_2(U_1)\oplus H_2(U_2) \rightarrow \\ \rightarrow &H_2(X) \rightarrow H_1(Y) \rightarrow H_1(U_1)\oplus H_1(U_2) \rightarrow H_1(X) \rightarrow 0.
\end{split} \end{align}
The last three terms in this sequence show that
\begin{equation}\label{eq:b1_bound}
b_1(U_1)+b_1(U_2)\leq b_1(Y)+b_1(X),
\end{equation}
with equality if and only if the map induced by the inclusions 
\[(i_1)_* \oplus (i_2)_* \colon H_1(Y) \rightarrow H_1(U_1)\oplus H_1(U_2)\]
is injective.

Since the Euler characteristic of an exact sequence is zero, \eqref{eq:MV} shows that 
\begin{equation}\label{eq:b2_formula}
b_2(X)=2b_1(X) + \sum_{i=1}^2 (b_2(U_i)-b_1(U_i) - b_3(U_i)),
\end{equation}
where we also used that $b_1(Y)=b_2(Y)$ and $b_1(X)=b_3(X)$.

By Novikov additivity, we have that $\sigma(X)=\sigma(U_1)+\sigma(U_2)$. So by summing the inequalities in \eqref{eq:sigbound} for $i=1,2$ and comparing with \eqref{eq:b2_formula} we obtain
\begin{equation}\label{eq:b2_bound}
  b_2(X) \ge 2(b_1(X) + b_2(Y) - b_1(U_1) - b_1(U_2)) + \sigma(X),
\end{equation}
with equality if and only if we have equality in \eqref{eq:sigbound} for both $i=1,2$.
Hence, $X$ can be positive definite if and only if 
\begin{equation}\label{eq:b1_bound2}
b_1(U_1)+b_1(U_2)= b_2(Y)+b_1(X).
\end{equation}
and we have equality in \eqref{eq:sigbound} for $i=1,2$. However we have already seen that equality occurs in \eqref{eq:b1_bound} if and only if $(i_1)_* \oplus (i_2)_*$ is injective. 
%
%
\end{proof}

This allows us to prove the main theorem.

\sfsineq*
\begin{proof} Let $X$ be the standard positive (semi-)definite plumbing $4$-manifold with $\partial X = Y$, and let $Z = X \cup_{Y} -W$. It follows from Proposition \ref{thm:def_gluing_thm} that $Z$ is positive definite. To see that condition \ref{enum:sigeq} of Proposition \ref{thm:def_gluing_thm} holds, note that from the homology long exact sequence for $(Y, W)$ we get
  \begin{equation}\label{eq:lesp}
    0 \rightarrow H_1(Y;\Q) \rightarrow H_1(W;\Q) \rightarrow H_1(W,Y;\Q) \rightarrow 0,
  \end{equation}
  where we used that $H_1(Y;\Q) \rightarrow H_1(W;\Q)$ is injective. Taking the Euler characteristic of \eqref{eq:lesp} and applying Lefschetz duality shows that $b_3(W) - b_1(W) + b_2(Y) = 0$. Thus, $Z$ is a smooth positive definite $4$-manifold, so by Donaldson's theorem $Z$ has standard positive diagonal intersection form. The inclusion $X \subset Z$ induces a map $H_2(X) \rightarrow H_2(Z)$ which preserves the intersection pairing. Thus, there is a $\Z$-linear map $(H_2(X), Q_X) \rightarrow (\Z^m, \mbox{Id})$ for some $m > 0$, which preserves intersection pairings.

  We construct a partition of $\{1,2,\ldots,k\}$ into at most $e$ classes as follows. Denote the orthonormal basis of coordinate vectors of $(\Z^m, \mbox{Id})$ by $\{e_1,\ldots,e_m\}$. For $v \in (H_2(X), Q_X)$, we call $\{e_i : 1\le i\le m, e_i \cdot v \neq 0\}$ the support of $v$. Without loss of generality, we may assume that the central vertex has support $\{e_1,e_2,\ldots, e_n\}$ where $n \le e$. Let $v_1,v_2,\ldots, v_k$ be the vertices of the plumbing adjacent to the central vertex, so that $v_i$ is a vertex belonging to the $i$th leg of the plumbing graph (with fraction $\frac{p_i}{q_i}$). For $i\in\{1,\ldots, n\}$, let $B_i = \{ 1 \le j \le k \mid v_j \cdot e_i \neq 0 \}$ and define $B_0 = \emptyset$. Let $C_i = B_i \backslash \cup_{j < i} B_j$ for $i \in \{1,\ldots,n\}$. Then $C_1,\ldots, C_n$ are disjoint and $\cup_i C_i = \{1,\ldots,k\}$. Thus the non-empty classes $\{C_i : C_i \neq \emptyset\}$ form a partition of $\{1,2,\ldots,k\}$ into at most $e$ classes. By definition for each $i \in \{1,2,\ldots, n\}$, the starting vertices of the linear chains indexed by $C_i$ all have support containing the common unit vector $e_i$. Hence, by Theorem \ref{thm:lattice_ineq}, we have that $\sum_{j \in C_i} \frac{q_j}{p_j} \le 1$.
\end{proof}
\section{Neumann-Zagier's question}\label{sec:nz}
We prove Theorem \ref{thm:nz_conj} below which, when combined with Donaldson's theorem, immediately implies Theorem \ref{thm:nz_simple}. Note that the following theorem also positively answers Neumann-Zagier's question stated in the introduction.


\begin{restatable}{thm}{nzconjthm} \label{thm:nz_conj} Let $Y = S^2(e; \frac{p_1}{q_1}, \ldots, \frac{p_k}{q_k})$, $k \ge 3$, be in standard form, that is, with $\frac{p_i}{q_i} > 1$ for all $i\in\{1,2,\ldots,k\}$, $e > 0$ and with $Y$ bounding a smooth positive definite plumbing $X$. Suppose that $|H_1(Y)| \in \{1,2,3,5,6,7\}$ and the intersection lattice $(H_2(X), Q_X)$ embeds into a positive standard diagonal lattice. Then $e = 1$. \end{restatable}

\begin{proof}[Proof of Theorem \ref{thm:nz_conj}]

For sake of contradiction, assume that $e > 1$. We may apply Theorem \ref{thm:sfsineq}, noting that the existence of $W$ in the hypothesis of Theorem \ref{thm:sfsineq} is only required to ensure that there is a map of lattices of $(H_2(X), Q_X)$ into a positive standard diagonal lattice. Hence, there is a partition $\{C_1,\ldots, C_n\}$ of $\{1,\ldots,k\}$ into $n \le e$ classes. Moreover, for each class $C$, $1 - \sum_{i \in C} \frac{q_i}{p_i} \ge 0$, and we call $C$ \emph{complementary} if equality occurs, and \emph{non-complementary} otherwise.
  
 We have 
  \begin{align}
    |H_1(Y)| &= p_1\cdots p_k \cdot \varepsilon(Y) = p_1\cdots p_k (e - \sum_{i=1}^k \frac{q_i}{p_i}) \nonumber \\
             &= p_1 \cdots p_k \left( (e - n) + \sum_{i=1}^n (1 - \sum_{j \in C_i} \frac{q_j}{p_j}) \right) \nonumber \\
             &= p_1 \cdots p_k (e - n) + \sum_{i=1}^n a_i \prod_{\substack{1 \le l \le k\\ l \not\in C_i}} p_l, \label{eqn:h1decomp}
  \end{align}
  where $a_i = (\prod_{j \in C_i} p_j) \cdot (1 - \sum_{l\in C_i} \frac{q_i}{p_i})$ is an integer for all $i \in \{1,\ldots,n\}$.
  Notice that all terms in \eqref{eqn:h1decomp} are non-negative integers. Since we are assuming that $|H_1(Y)| \in \{1,2,3,5,6,7\}$, we must have $n = e$, otherwise $|H_1(Y)| \ge p_1 \cdots p_k \ge 2\cdot 2\cdot 2 = 8$ since $k \ge 3$.
  
We claim that $|C_i| \le k-2$ for some $i \in \{1,\ldots,e\}$ with $C_i$ non-complementary. To see this, we argue as follows. There are $n = e \ge 2$ classes in the partition, and at least one non-complementary class since $|H_1(Y)| > 0$. If there are two non-complementary classes then at least one has size at most $k-2$ since $k \ge 3$. If there is only one non-complementary class, then there is a complementary class which necessarily has size at least $2$, and hence the non-complementary class satisfies the claim.

Combining the above claim with \eqref{eqn:h1decomp}, we see that $|H_1(Y)|$ is a sum of integers greater than $1$, and at least one of these integers is not prime. For $|H_1(Y)| \in \{1,2,3,5,6,7\}$, this is only possible for $|H_1(Y)| = 7$ with decomposition $7 = 3 + 2 \cdot 2$, and for $|H_1(Y)| = 6$ with the two decompositions $|H_1(Y)| = 2 + 2\cdot 2 = 2\cdot 3$. We address these cases in turn.
For $|H_1(Y)| = 7 = 3 + 2 \cdot 2$, comparing this decomposition with \eqref{eqn:h1decomp}, we see that there must exist some non-complementary $C_i$ with $|C_i| = 2$ and $p_j = 2$ for all $j \in C_i$. However, such a $C_i$ must be complementary since $1 - \frac{1}{2} - \frac{1}{2} = 0$, a contradiction. A similar argument rules out the decomposition $|H_1(Y)| = 2 + 2\cdot 2$. Finally, in the case $|H_1(Y)| = 2\cdot 3$, the decomposition implies that there exists a complementary class $C_i = \{a,b\}$ with $p_a = 2$ and $p_b = 3$, which is impossible.

\end{proof}
We obtain the following corollary, answering a question of Lidman-Tweedy \cite[Remark 4.3]{1707.09648}.

\coroldinv*
\begin{proof} We prove the contrapositive. Assume that $d(Y) = 0$. Note that reversing the orientation of $Y$ simply changes the sign of the weight of the central vertex in the definite plumbing bounding $Y$. Thus, by reversing the orientation of $Y$ if necessary we assume that $Y$ bounds a smooth negative definite plumbing $X^4$. Let $\mathcal{C} = \{\xi \in H_2(X;\Z) \mid \xi \cdot v = v \cdot v \pmod{2} \mbox{ for all }v\in H_2(X;\Z)\}$ be the set of characteristic vectors, and let $n = \mbox{rk}(H_2(X))$. Elkies \cite{MR1338791} proved that
  $0 \le n + \max_{\xi \in \mathcal{C}} \xi \cdot \xi$, with equality if and only if $Q_X$ is diagonalizable over $\Z$. However, it follows from \cite[Theorem 9.6]{MR1957829} that $n + \max_{\xi \in \mathcal{C}} \xi \cdot \xi \le 4d(Y) = 0$. Therefore $Q_X$ is diagonalizable over $\Z$, in particular $(H_2(-X), Q_{-X})$ embeds into a positive standard diagonal lattice. Hence, Theorem \ref{thm:nz_conj} implies that $|e| = 1$.
\end{proof}

\section{Seifert fibered spaces bounding rational homology $S^1 \times D^3$'s}\label{sec:bding_qhs1xd3}
In this section we prove Theorem \ref{thm:bd_qhs1xd3}, which in particular gives a classification of the Seifert fibered spaces which smoothly bound rational homology $S^1\times D^3$'s. We note that the implication \eqref{enum:bdqhs1d3} implies \eqref{enum:comppairs} was proved by Donald \cite[Proof of Theorem 1.3]{MR3271270}, and the equivalence of \eqref{enum:comppairs} and \eqref{enum:bdqhs1d3} was shown by Aceto \cite[Theorem 1.2]{1502.03863}.

\paolosthm*

\begin{proof} 
  First suppose that \eqref{enum:comppairs} holds, that is, $Y = S^2(k; \frac{p_1}{q_1}, \frac{p_1}{p_1-q_1}, \ldots, \frac{p_k}{q_k}, \frac{p_k}{p_k - q_k})$, where $k \ge 0$ and $\frac{p_i}{q_i}\in \Q_{>1}$ for all $i\in\{1,\ldots,k\}$. By Rolfsen twisting, $Y$ can be put into the form $S^2(0; \frac{p_1}{q_1}, -\frac{p_1}{q_1}, \ldots, \frac{p_k}{q_k}, -\frac{p_k}{q_k})$. Let $M = S^2(0; \frac{p_1}{q_1}, \frac{p_2}{q_2}, \ldots, \frac{p_k}{q_k})$, let $M^{\circ}$ be the $3$-manifold with torus boundary given by removing a tubular neighbourhood of a regular fiber of $M$ and let $W = M \times [0, 1]$. Then $\partial W = M^\circ \cup_{\partial} -M^\circ$ is the double of $M^\circ$, which is precisely $Y$. Finally, notice that $H_*(W; \Q) = H_*(M^{\circ}; \Q) = H_*(S^1 \times D^3; \Q)$, where the last equality follows from the fact that $M$ is a rational homology $S^3$ and $M^\circ$ is obtained by removing a neighbourhood of a simple closed curve from $M$. This proves \eqref{enum:bdqhs1d3}.
  
  The implication \eqref{enum:bdqhs1d3} implies \eqref{enum:bdposneg} holds by taking $W_{\pm} = W$ and noting that $H_1(Y;\Q) \rightarrow H_1(W;\Q)$ is injective by the long exact sequence of the pair $(Y, W)$.
  
  Finally assume that \eqref{enum:bdposneg} holds. Hence, $Y$ is the boundary of smooth $4$-manifolds $W_+$ and $W_-$ satisfying $\sigma(W_\pm) = \pm b_2(W_\pm)$ and such that the inclusion induced maps $H_1(Y) \rightarrow H_1(W_\pm)$ are injective. Write $Y$ as $S^2(e; \frac{p_1}{q_1}, \ldots, \frac{p_k}{q_k})$ with $k \ge 3$ and $\frac{p_i}{q_i} \in \Q_{>1}$ for all $i\in\{1,\ldots,k\}$. Notice that $-Y = S^2(k-e; \frac{p_1}{p_1 - q_1}, \ldots, \frac{p_k}{p_k - q_k})$, and $Y$ is of the form given in \eqref{enum:comppairs} if and only if $-Y$ is of this form.
  Thus, by reversing the orientations of both $Y$ and $W_{\pm}$ if necessary, we may assume that $e \ge \frac{k}{2}$.

  By Theorem \ref{thm:sfsineq}, there is a partition $\{C_1,\ldots, C_n\}$ of $\{1,\ldots,k\}$ into $n \le e$ classes such that for each class $C$, $1 - \sum_{i \in C} \frac{q_i}{p_i} \ge 0$. Since $Y$ is a rational homology $S^1\times S^2$, we thus have
    \begin{align*}
      0 &= p_1\cdots p_k \cdot \varepsilon(Y) = p_1\cdots p_k (e - \sum_{i=1}^k \frac{q_i}{p_i}) \\
                 &= p_1 \cdots p_k \left( (e - n) + \sum_{i=1}^n (1 - \sum_{j \in C_i} \frac{q_j}{p_j}) \right),
    \end{align*}
    where all terms in the sum are non-negative. Hence, we must have $n = e$ and $1 - \sum_{i \in C} \frac{q_i}{p_i} = 0$, for all $i\in\{1,\ldots,n\}$. This implies that $|C_i| \ge 2$ for all $i\in\{1,\ldots,n\}$. Thus, there are at least $2n = 2e$ legs, so $e \le \frac{k}{2}$. However, by assumption $e \ge \frac{k}{2}$ so $e = \frac{k}{2}$ and $|C_i| = 2$ for all $i\in\{1,\ldots,n\}$. Thus, $C_1,\ldots,C_n$ partition $\{1,\ldots,k\}$ into pairs of indices indexing pairs of fractions of the form $\frac{p}{q}, \frac{p}{p-q} \in \Q_{>1}$, and thus \eqref{enum:comppairs} holds.
\end{proof}
\phantomsection
\bibliography{references}{}
\bibliographystyle{alpha}

\end{document}